\documentclass[a4paper,12pt]{article}
\usepackage{amsmath,amsfonts,amssymb,color}
\usepackage[english]{babel}
\usepackage{amsmath,amsfonts,amssymb,amscd,mathrsfs}
\usepackage{amsfonts}
\usepackage{amsthm}

\newtheorem{teo}{Theorem}[section]
\newtheorem{prop}[teo]{Proposition}
\newtheorem{lem}[teo]{Lemma}
\newtheorem{coro}[teo]{Corollary}

\newtheorem{rem}[teo]{Remark}

\begin{document}

\title{\vspace*{0cm}Riemannian metrics on an infinite dimensional symplectic group}

\date{}
\author{Manuel L\'opez Galv\'an\footnote{Supported by Instituto Argentino de Matem\'atica(CONICET-PIP 2010-0757), Universidad Nacional de General Sarmiento and ANPCyT (PICT 2010-2478).}}

\maketitle

\setlength{\parindent}{0cm} 

\begin{abstract}
The aim of this paper is the geometric study of the symplectic operators which are a perturbation of the identity by a Hilbert-Schmidt operator. This subgroup of the symplectic group was introduced in Pierre de la Harpe's classical book of Banach-Lie groups. Throughout this paper we will endow the  tangent spaces with different Riemannian metrics. We will use the minimal curves of the unitary group and the positive invertible operators to compare the length of the geodesic curves in each case. Moreover we will study the completeness of the symplectic group with the geodesic distance. 
\end{abstract}

\section{Introduction}
The symplectic group has many applications in quantum theory with infinitely many degrees of freedom, i.e. in canonical quantum field theory,
string theory, statistical quantum physics and solition theory. According to Shale's definitions \cite{Shale}, if we have a quantization $R(.)$ of the real Hilbert space $\Sigma(\mathcal{H})$ it is of interest to determine the subgroups of the symplectic group consisting of those $g$ for which exists an unitary transformation $Y(g)$ such that $R(gz)=Y(g)R(z)Y(g)^{-1}$ for all $z \in \mathcal{H}$. Let $\vert g\vert :=(g^*g)^{1/2}$ be the absolute value operator, in \cite{Shale} it was proved that in the case of Fock-Cook quantization (see \cite{Cook} for some background) the subgroup is $\lbrace g : \vert g\vert -1  \ \mbox{is Hilbert-Schmidt}\rbrace$. 

In this paper we study a variant of this subgroup, in which $g$ is only a perturbation of the identity by a Hilbert-Schmidt operator. In classical finite dimensional Riemannian theory it is well known the fact that given two points there is a minimal geodesic curve that joins them and this is equivalent to the completeness of the metric space with the geodesic distance; this is the Hopf-Rinow theorem. In the infinite dimensional case this is no longer true. In \cite{McAlpin} and \cite{Atkin}, McAlpin and Atkin showed in two examples how this theorem can fail. The main result of this paper establishes that if we consider the left invariant metric in the restricted symplectic group then its geodesic distance makes of the group a complete metric space. In the process to do it, we use the existence of a smooth polar decomposition in the group; this will allow us to define a mixed metric related to the unitary and positive part of the group. In this way we will use minimality results of the restricted unitary group $U_2(\mathcal{H})$ (see \cite{Andruchow2}) and we also prove some geometric properties of the symplectic positive operators with different Riemannian metrics.       

\section{Background and definitions}
Let $\mathcal{H}$ be an infinite dimensional real Hilbert space and let $\mathcal{B}(\mathcal{H})$ be the space of bounded operators. Denote by $\mathcal{B}_2(\mathcal{H})$ the Hilbert-Schmidt class $$\mathcal{B}_2(\mathcal{H})=\left\{a \in \mathcal{B}(\mathcal{H}): Tr(a^*a)< \infty\right\}$$ where $Tr$ is the usual trace in $\mathcal{B}(\mathcal{H})$. This space is a Hilbert space with the inner product $$<a,b>=Tr(b^*a).$$ The norm induced by this inner product is called the 2-norm and denoted by $$\Vert a \Vert_2=Tr(a^*a)^{1/2}.$$ The usual operator norm will be denoted by $\Vert \ \Vert$.

If $\mathcal{A}\subset\mathcal{B}(\mathcal{H})$ is any subset of operators we use the subscript $h$ (resp $ah$) to denote the subset of Hermitian (resp. anti-Hermitian) operators of it, i.e. $\mathcal{A}_h=\left\{ x \in \mathcal{A}: x^*=x\right\}$ and $\mathcal{A}_{ah}=\left\{ x \in \mathcal{A}: x^*=-x\right\}$. 

We fix a complex structure; that is a linear isometry $J \in \mathcal{B}(\mathcal{H})$ such that, $$J^2=-1  \ \mbox{and} \ J^*=-J.$$
The symplectic form $w$ is given by $w(\xi,\eta)=\left\langle J\xi,\eta \right\rangle$. 

We denote by $GL(\mathcal{H})$ the group of invertible operators, with $GL(\mathcal H)^+$ the space of positive invertible operators,  and by ${\rm Sp}(\mathcal{H})$ the subgroup of invertible operators which preserve the symplectic form, that is $g \in {\rm Sp}(\mathcal{H})$ if $w(g\xi,g\eta)=w(\xi,\eta)$. Algebraically $${\rm Sp}(\mathcal{H})=\left\{ g \in GL(\mathcal{H}): g^*Jg=J\right\}.$$ This group is a Banach-Lie group and its Banach-Lie algebra is given by $$\mathfrak{sp}(\mathcal{H})=\left\{x \in \mathcal{B}(\mathcal{H}): xJ=-Jx^*\right\}.$$ Denote by $\mathcal{H}_J$ the Hilbert space $\mathcal{H}$ with the action of the complex field $\mathbb{C}$ given by $J$, that is; if $\lambda=\lambda_1+i\lambda_2 \in \mathbb{C}$ and $\xi \in \mathcal{H}$ we can define the action as $\lambda\xi:=\lambda_1\xi+\lambda_2J\xi$ and the complex inner product as $<\xi,\eta>_{\mathbb{C}}=<\xi,\eta>-iw(\xi,\eta)$. 

Denote by $\mathcal{B}(\mathcal{H}_J)$ the space of bounded complex linear operators in $\mathcal{H}_J$. A straightforward computation shows that $\mathcal{B}(\mathcal{H}_J)$ consists of the elements of $\mathcal{B}(\mathcal{H})$ which commute with $J$.   

One property that we will use in this paper is the stability of the adjoint operation. We give a short proof of this fact.
\begin{prop}\label{est adjunta} If $g \in {\rm Sp}(\mathcal{H})$ then $g^* \in {\rm Sp}(\mathcal{H})$.
\end{prop}
\begin{proof} The proof is a short computation using the definition, indeed if $g \in {\rm Sp}(\mathcal{H})$ then $g^*J=Jg^{-1}$ and times by $gJ$ we obtain $gJg^*J=-1$ then $gJg^*=J$.   
\end{proof}
The above proposition leads us to one of the most important properties of the symplectic group, that is the stability under polar decompositions. 
\begin{coro} If $u\vert g\vert$ is the polar decomposition of an element in ${\rm Sp}(\mathcal{H})$ then its unitary part $u$ and its positive part $\vert g\vert$ belong in ${\rm Sp}(\mathcal{H})$.
\end{coro}
 We consider now the restricted subgroup of ${\rm Sp}(\mathcal{H})$ $${\rm Sp}_2(\mathcal{H})=\left\{ g \in {\rm Sp}(\mathcal{H}): g-1 \in \mathcal{B}_2(\mathcal{H})\right\}.$$ Since $\mathcal{B}_2(\mathcal{H})$ is a Banach algebra, it is clear that this subgroup is also stable for the adjoint operation and for the polar decomposition; we denote by $$U_2(\mathcal{H}_J)=\left\{g \in U_2(\mathcal{H}): gJ=Jg\right\} \ \ \mbox{and}$$ $${\rm Sp}_2^+({\mathcal H})=\left\{g \in {\rm Sp}_2(\mathcal{H}) : g>0\right\}$$  its unitary part and positive part respectively where $U_2(\mathcal{H})$ is the classical unitary group whose elements are Hilbert-Schmidt perturbations of the identity operator 1. It is obvious that the unitary part is a closed subgroup of ${\rm Sp}_2(\mathcal{H})$. In the infinite dimensional setting, this does not guarantee a nice submanifold structure; in Proposition \ref{subgr} we will prove that $U_2(\mathcal{H}_J)$ is a Banach-Lie subgroup of ${{\rm Sp}_2}(\mathcal{H})$.
 
 Throughout this paper, if $M$ is any submanifold of the symplectic group, we will denote by  $\mathfrak{b}(g,v)$ the metric in each tangent space  $T_gM$. The length of a smooth curve measured with the metric $\mathfrak{b}$ will be denoted by $$L_{\mathfrak{b}}(\alpha)=\int_0^1 \mathfrak{b}(\alpha(t),\dot{\alpha}(t)) dt.$$ We define the geodesic distance between two points $p,q \in M\subseteq {\rm Sp}_2(\mathcal{H})$ as the infimum of the length of all piecewise smooth curves in $M$ joining $p$ to $q$,  $$d_{\mathfrak{b}}(p,q)=\inf\left\{L_{\mathfrak{b}}(\alpha): \alpha\subset M, \alpha(0)=p, \ \alpha(1)=q \right\}.$$

\section{Local structure of ${\rm Sp}_2(\mathcal{H})$}
Some of the following facts are have been well-known for general Schatten ideals, more precisely the Banach-Lie group structure was noted in the book \cite{Harpe}. Here we will complete some details for our case of the Hilbert-Schmidt ideal.

Given $g_1,g_2 \in {\rm Sp}_2(\mathcal{H})$, it is obvious that $g_1-g_2$ belongs in $\mathcal{B}_2(\mathcal{H})$; hence we can endow the restricted symplectic group with the metric  $\|g_1-g_2\|_2$.   

\begin{prop} \label{sp2cerrado} The metric space $({\rm Sp}_2(\mathcal{H}), \Vert . \Vert_2)$ is complete.
\end{prop}
\begin{proof} Let $(x_n)\subset {\rm Sp}_2(\mathcal{H})$ be a Cauchy sequence, then $x_n-1$ is a Cauchy sequence in $\mathcal{B}_2(\mathcal{H})$. From this, we can take $x \in \mathcal{B}_2(\mathcal{H})$ such that  $x_n\longrightarrow 1+x:=x_0$ in $\|.\|_2$. It is clear that $x_0$ verifies the algebraic relation $x_0^{*}Jx_0=J$; to complete the proof we will see that $x_0$ is invertible. Indeed, from $x_n^{*} \in {\rm Sp}_2(\mathcal{H})$ we have  $x_nJx_n^{*}=J$, then this relation is transferred through the limit to $x_0$. We can now define the inverse of $x_0$ as $x_0^{-1}:=-Jx_0^{*}J$, it verifies:   
$$x_0^{-1}x_0=-Jx_0^{*}Jx_0=1  \ \ \mbox{and}  \  \  x_0x_0^{-1}=x_0(-Jx_0^{*}J)=1.$$
\end{proof}

\subsection{Differentiable structure}

Now we will show that ${\rm Sp}_2(\mathcal{H})$ has differentiable structure. Let us denote $$\mathfrak{sp}_2(\mathcal{H})=\left\{x \in \mathcal{B}_2(\mathcal{H}): xJ=-Jx^*\right\}.$$ It is clear that $\mathfrak{sp}_2(\mathcal{H})$ is a Banach-Lie subalgebra of $\mathcal{B}_2(\mathcal{H})$. 

\begin{lem}${\rm Sp}_2(\mathcal{H})$ is a Banach-Lie group.
\end{lem}
\begin{proof}Let $\exp: \mathcal{B}_2(\mathcal{H}) \rightarrow 1+\mathcal{B}_2(\mathcal{H})$ be $\exp(x)=e^x$ the exponential map. If we compute the exponential on $\mathfrak{sp}_2(\mathcal{H})$, its image belongs in ${\rm Sp}_2(\mathcal{H})$. Indeed, if $x$ verifies $xJ=-Jx^*$ then $e^xJ=Je^{-x^*}=J(e^{x^*})^{-1}$ and thus $e^{x}Je^{x^*}=J$; moreover we have that $e^x-1=x+x^2/2+... \in \mathcal{B}_2(\mathcal{H})$. Since the exponential is a local diffeomorphism ($d_0\exp=Id$) there exists $r<1$ such that $$U=\left\{x=\log(g):\Vert g-1\Vert_2<r\right\}\stackrel{\exp}\longrightarrow V=\left\{g \in 1+\mathcal{B}_2(\mathcal{H}): \Vert g-1\Vert_2<r\right\}$$ 
is an analytic diffeomorphism.

On the other hand if $g\in {\rm Sp}_2(\mathcal{H})$ meets $\Vert g-1\Vert_2<r$  ($r<1$) the exponential is a diffeomorphism and then its inverse is given by the logarithmic series $x=\log(g)=\sum_{n=1}^{\infty}(-1)^n\frac{(1-g)^n}{n} \in \mathcal{B}_2(\mathcal{H})$ and satisfies the condition $xJ=-Jx^*$. 
Therefore $\exp$ is one to one between $U\cap \mathfrak{sp}_2(\mathcal{H})$ and $V\cap {\rm Sp}_2(\mathcal{H})$. We have found a local chart around 1. This construction can be translated to any point in ${\rm Sp}_2(\mathcal{H})$ using the left action of ${\rm Sp}_2(\mathcal{H})$ on itself.
\end{proof}

Since the exponential map $\exp : \mathcal{B}_2(\mathcal{H}_J)_{ah} \rightarrow U_2(\mathcal{H}_J)$ is surjective (see \cite{Andruchow2}), it is clear that $\exp(\mathfrak{sp}_2(\mathcal{H})_{ah})=U_2(\mathcal{H}_J)$.
\begin{prop}\label{subgr} The unitary subgroup $U_2(\mathcal{H}_J)$ is a Lie-subgroup of ${\rm Sp}_2(\mathcal{H})$.
\end{prop}
\begin{proof} Let $U$ be a neighboord of 0 in $\mathfrak{sp}_2(\mathcal{H})$ such that the exponential map is a diffeomorphism, we can assume that $U=\left\lbrace x \in \mathfrak{sp}_2(\mathcal{H}): \| x\|_2<r\right\rbrace $ for a suitable $r>0$. It is clear that we always have $$\exp( \mathfrak{sp}_2(\mathcal{H})_{ah}\cap U)\subseteq U_2(\mathcal{H}_J)\cap \exp(U).$$
Conversely, suppose that $g\in U_2(\mathcal{H}_J)\cap \exp(U)$ then $g=e^y$ for some $y \in U$; hence $1=gg^*=e^ye^{y^*}$ and then $e^y=e^{-y^*}$. Since $-y^{*}$ also belongs in $U$ and the exponential is one to one, we have that $y=-y^*$ and thus $y\in \mathfrak{sp}_2(\mathcal{H})_{ah}$. Then we have $\exp( \mathfrak{sp}_2(\mathcal{H})_{ah}\cap U)=U_2(\mathcal{H}_J)\cap \exp(U)$ and this implies that $U_2(\mathcal{H}_J)$ is a Lie-subgroup of ${\rm Sp}_2(\mathcal{H})$ (see Prop. 4.4 in the book \cite{Beltita}).

\end{proof}

\section{The left invariant metric of  ${\rm Sp}_2(\mathcal{H})$}
Again, using the left action of ${\rm Sp}_2(\mathcal{H})$ on itself, the tangent space at $g \in {\rm Sp}_2(\mathcal{H})$ is $$(T{\rm Sp}_2(\mathcal{H}))_g=g.\mathfrak{sp}_2(\mathcal{H})\subset \mathcal{B}_2(\mathcal{H}).$$
We introduce the left invariant metric for $v \in (T{\rm Sp}_2(\mathcal{H}))_g$ by $$\mathcal{I}(g,v):=\Vert g^{-1}v\Vert_2.$$ This metric comes from the inner product $$\left\langle v,w\right\rangle_g=\left\langle g^{-1}v,g^{-1}w\right\rangle=tr((gg^*)^{-1}vw^*).$$
In the followings steps we recall the metric spray of $GL_2(\mathcal{H})$ with the left invariant metric, for more details see \cite{Andruchow1}. We will follow the notation of Lang's book \cite{Lang}. For the metric expression $g \longmapsto I_g$ where $I_gv=(gg^*)^{-1}v$ we obtain the metric spray (see \cite{Larotonda2})
$$F_g(v)=vg^{-1}v+gv^*I_gv-vv^*(g^*)^{-1}.$$ Using the polarization formula $$\Gamma_g(v,w)=1/2\left\{F_g(v+w)-F_g(v)-F_g(w)\right\}$$ we obtain the bilinear form associated to the spray, that is for $g\in GL_2(\mathcal{H})$ and $v=gx$,$ \ w=gy \in T_gGL_2(\mathcal{H})$, $$2g^{-1}\Gamma_g(gx,gy)=xy+yx+x^*y+y^*x-xy^*-yx^*.$$ The covariant derivative of the spray is $D_t\eta=\dot{\eta}-\Gamma(\eta,\dot{\alpha})$ where $\alpha:(-\epsilon,\epsilon) \rightarrow GL_2(\mathcal{H})$ is any smooth curve and $\eta$ is a tangent field along $\alpha$.  
\begin{prop} If $\eta$ is a field along a curve $\alpha$ we define $\beta=\alpha^{-1}\dot{\alpha}$ and $\mu=\alpha^{-1}\eta$, the fields at the identity, then the covariant derivate can be expressed by $$\alpha^{-1}D_t\eta=\dot{\mu}+1/2\lbrace [\beta,\mu]+[\beta,\mu^{*}]+[\mu,\beta^{*}]\rbrace.$$   
\end{prop}
\begin{proof} From the covariant derivate formula, we have $$\alpha^{-1}D_t\eta=\alpha^{-1}\dot{\eta}-\alpha^{-1}\Gamma(\alpha \mu,\alpha\beta).$$ If we write $\eta=\alpha\mu$ and $\dot{\alpha}=\alpha\beta$, using the product rule to differentiate $\eta$ we obtain $$\alpha^{-1}\dot{\eta}=\alpha^{-1}\dot{\alpha}\mu+\dot{\mu}=\beta\mu+\dot{\mu}.$$ 
\end{proof}
The above formula can be restricted to ${\rm Sp}_2(\mathcal{H})$ preserving the tangent fields, that is, ${\rm Sp}_2(\mathcal{H}) \subset GL_2(\mathcal{H})$ is totally geodesic.
\begin{prop}If $\alpha \subset {\rm Sp}_2(\mathcal{H})$ is a curve and $\eta$ a field along $\alpha$ then $$D_t\eta \in (T{\rm Sp}_2(\mathcal{H}))_\alpha=\alpha.\mathfrak{sp}_2(\mathcal{H}).$$
 \end{prop}
\begin{proof} Let $\beta=\alpha^{-1}\dot{\alpha}$ and $\mu=\alpha^{-1}\eta$ be the fields moved to $\mathfrak{sp}_2(\mathcal{H})$, we will show that $\alpha^{-1}D_t\eta \subset \mathfrak{sp}_2(\mathcal{H})$. Indeed $\mu$ verifies $\mu J=-J\mu^*$, if we derive, we obtain $\dot{\mu}J=-J\dot{\mu}^*$ and  $\dot{\mu}$ is a Hilbert-Schmidt operator and that lies in $\mathfrak{sp}_2(\mathcal{H})$. The brackets $[\beta,\mu],  [\beta,\mu^{*}],  [\mu,\beta^{*}]$ are all in $\mathfrak{sp}_2(\mathcal{H})$ since it is a Banach-Lie algebra, then using the above proposition $$\alpha^{-1}D_t\eta= \dot{\mu}+1/2\lbrace [\beta,\mu]+[\beta,\mu^{*}]+[\mu,\beta^{*}]\rbrace \subset \mathfrak{sp}_2(\mathcal{H}).$$    
\end{proof}

This shows that the Riemannian connection given by the left invariant metric in the group ${\rm Sp}_2(\mathcal{H})$ matches the one of $GL_2(\mathcal{H})$. Particularly the geodesics of ${\rm Sp}_2(\mathcal{H})$ are the same than those $GL_2(\mathcal{H})$; if $g_0 \in {\rm Sp}_2(\mathcal{H})$ and $g_0v_0 \in g_0.\mathfrak{sp}_2(\mathcal{H})$ are the initial position and the initial velocity then  $$\alpha(t)=g_0e^{tv_0^{*}}e^{t(v_0-v_0^{*})}\subset {\rm Sp}_2(\mathcal{H})$$ satisfies $D_t\dot{\alpha}=0$ (see \cite{Andruchow1}). In this context the Riemannian exponential for $g \in {\rm Sp}_2(\mathcal{H})$ is $$Exp_g(v)=ge^{v^{*}}e^{v-v^{*}}$$ with $v \in \mathfrak{sp}_2(\mathcal{H})$. 

\section{Metric structure in ${\rm Sp}_2^+(\mathcal{H})$}
\subsection{${\rm Sp}_2^+({\mathcal H})$ as a submanifold of $GL({\mathcal H})^+$}
It is not difficult to prove using the functional calculus that the exponential map can be restricted to the Lie-algebra $\mathfrak{sp}_2(\mathcal{H})_h$ making it diffeomorphic to ${\rm Sp}_2^+({\mathcal H})$; in this way $$\exp: \mathfrak{sp}_2(\mathcal{H})_h\longrightarrow {\rm Sp}_2^+({\mathcal H})$$ is a diffeomorphism.
From the stability of the adjoint operation in ${\rm Sp}_2(\mathcal{H})$ (Proposition \ref{est adjunta}) we can restrict the natural action of the invertible group to the set of positive invertible operators. 
\begin{lem} \label{acpos}The natural action $l: {\rm Sp}_2(\mathcal{H}) \times {\rm Sp}_2^+(\mathcal{H}) \longrightarrow {\rm Sp}_2^+({\mathcal H})$ given by $$(g,a) \longmapsto gag^{*}$$ is well defined and transitive. 
\end{lem}
\begin{proof} The map $(g,a) \longmapsto gag^{*}$ is well defined as a direct consequence of Proposition \ref{est adjunta}; it is clear that $gag^* \in {\rm Sp}_2(\mathcal{H})$ and it is positive. If $X,Y \in {\rm Sp}_2^+({\mathcal H})$, we can assume that $X=e^x, \ Y=e^y$ where $x,y \in {\mathfrak{sp}_2(\mathcal{H})_h}$; then if we consider the operator $g=e^{x/2}e^{-y/2} \in {\rm Sp}_2(\mathcal{H})$ it verifies that $X=gYg^*$.
\end{proof}
Now we endow the closed submanifold ${\rm Sp}_2^+({\mathcal H})$ with a Riemannian metric; if $a \in {\rm Sp}_2^+({\mathcal H})$ and $x \in T_a{\rm Sp}_2^+({\mathcal H})=\left\{a^{1/2}\ln(a^{-1/2}qa^{-1/2})a^{1/2}: q \in {\rm Sp}_2^+({\mathcal H})\right\}$ we put the metric of positive operators (see \cite{Corach1} and \cite{Mostow}) given by $$\mathfrak{p}(a,x):=\|a^{-1/2}xa^{-1/2}\|_2.$$ 

\begin{rem} The above metric is invariant under the action of the group ${\rm Sp}_2(\mathcal{H})$, that is: if $x\in T_a{\rm Sp}_2^+({\mathcal H})$ then 
 $$\mathfrak{p}(gag^{*},gxg^{*})=\mathfrak{p}(a,x).$$
\end{rem}
 
The curve $\gamma_{pq}(t)=p^{1/2}(p^{-1/2}qp^{-1/2})^{t}p^{1/2}=p^{1/2}e^{t(\ln(p^{-1/2}qp^{-1/2}))}p^{1/2} \subset{\rm Sp}_2^+({\mathcal H})$ joins $p$ to $q$ and its length is $$L_{\mathfrak{p}}(\gamma_{pq})=\|\ln(p^{-1/2}qp^{-1/2})\|_2.$$ This curve is minimal among all curves in ${\rm Sp}_2^+({\mathcal H})$ that join $p$ to $q$. We will give a short proof of this fact, the key is the following inequality.

\begin{rem}(See \cite{Hiai}) If $d\exp_x$ denotes the differential of exponential at $x$ of the usual exponential map, then 
\begin{equation}
\mathfrak{p}(e^x ,d \exp_x(y))=\Vert e^{-x/2}d \exp_x(y)e^{-x/2}\Vert_2\geq \Vert y\Vert_2. \label{EMI} 
\end{equation}
for any $x,y\in \mathcal{B}_2(\mathcal{H})_h$.
\end{rem}

\begin{teo}\label{minpos} Let $p,q \in {\rm Sp}_2^+({\mathcal H})$ then $\gamma_{pq}\subset {\rm Sp}_2^+({\mathcal H})$ has minimal length among all curves that joins $p$ to $q$. 
\end{teo}

\begin{proof}
We can suppose that $p=1$, then $\gamma_{1q}(t)=e^{tx}$ where $x=\ln(q)$ and its length is $\Vert x\Vert_2=\Vert \ln(q)\Vert_2$. If $\alpha$ is another curve that joins the same points, then it can be written as $\alpha(t)=e^{\beta(t)}$ where $\beta(t)=\ln(\alpha(t)) \subset \mathfrak{sp}_2(\mathcal{H})_h$. Using the above remark we have $$L_{\mathfrak{p}}(\gamma_{1q})=\|x-0\|_2=\| \int_0^{1} \dot{\beta}(t) dt \|_2 \leq \int_0^{1} \|\dot{\beta}(t)\|_2 dt$$ and also  
$$\mathfrak{p}(\alpha,\dot{\alpha})=\mathfrak{p}\big(e^{\beta(t)},d\exp_{\beta(t)}(\dot{\beta}(t))\big)$$
$$=\|e^{-\beta(t)/2} d\exp_{\beta(t)}(\dot{\beta}(t)) e^{-\beta(t)/2} \|_2 \geq \|\dot{\beta}(t)\|_2.$$
\end{proof}
It can be shown that the metric space $({\rm Sp}_2^+({\mathcal H}),d_{\mathfrak{p}})$ is complete. This fact was proved in \cite{Corach2} or \cite{Larotonda} in another context; in this context we also can derive from (\ref{EMI}) the known inequality  $$d_{\mathfrak{p}}(p,q) \geq \|\log p - \log q\|_2 $$ for $p,q \in {\rm Sp}_2^+({\mathcal H})$; the proof of completeness can be adapted easily, therefore we omit them.

\subsection{${\rm Sp}_2^+({\mathcal H})$ as submanifold of the ambient space}

Here we will think ${\rm Sp}_2^+({\mathcal H})$ as a submanifold of the real Hilbert space $H_\mathbb{R}:=\mathbb{R}\oplus\mathcal{B}_2(\mathcal{H})_h$ with the natural inner product $$<\lambda+a,\mu+b>=\lambda\mu +Tr(b^*a).$$ From the action given by Lemma \ref{acpos} we can define for each $a \in {\rm Sp}_2^+({\mathcal H})$ the map $$\pi_a : {\rm Sp}_2(\mathcal{H}) \rightarrow {\rm Sp}_2^+({\mathcal H}), \ \ \pi_a(g)=gag^{*}.$$ Observe that, since the action is transitive this map is onto and as in the case of the full space of positive invertible operators ${\mathcal{B}(\mathcal{H})}^+$(see \cite{Corach3}), we have that $\sigma_a(b)=b^{1/2}a^{-1/2}$ defines a global smooth section of $\pi_a$. Note that this map is well defined and its image belongs clearly to ${\rm Sp}_2(\mathcal{H})$.

 If $g$ is any element in ${\rm Sp}_2^+({\mathcal H})$, we can consider the real linear map $$\Pi_g : H_\mathbb{R} \longrightarrow H_\mathbb{R}, \ x\longmapsto \frac{1}{2}\big(x+gJxJg\big).$$This map is well defined and a short computation shows that the range belongs to $\mathcal{B}_2(\mathcal{H})_h$.     

\begin{lem} \label{adj} The map $\Pi_g$ is idempotent and its range is $g^{1/2}\mathfrak{sp}_2(\mathcal{H})_hg^{1/2}$. Moreover, its adjoint map for the trace inner product is $\Pi_{g^{-1}}$. If $g=1$ this map is the orthogonal projection onto $\mathfrak{sp}_2(\mathcal{H})_h$.
\end{lem}

\begin{proof}First we prove that $\Pi_g$ is an idempotent map. Indeed, using the fact that $gJg=J$,  $$\Pi^2_g(x)=\Pi_g(\frac{1}{2}\big(x+gJxJg\big))=\frac{1}{4}\big(x+gJxJg+gJ(x+gJxJg)Jg\big)=$$  
$$=\frac{1}{4}\big(x+2gJxJg+(gJg)JxJ(gJg)\big)=\Pi_g(x).$$ Now we will prove that $\mbox{Ran}(\Pi_g)=g^{1/2}\mathfrak{sp}_2(\mathcal{H})_hg^{1/2}$. Indeed, let $g^{1/2}xg^{1/2}$ with $x \in \mathfrak{sp}_2(\mathcal{H})_h$, then using that $g^{1/2}Jg^{1/2}=J$ (that is $g^{1/2} \in {\rm Sp}_2^+({\mathcal H})$) and the relation of $x$ with $J$ we have 
$$\Pi_g(g^{1/2}xg^{1/2})=\frac{1}{2}\big(g^{1/2}xg^{1/2}+g^{1/2}g^{1/2}Jg^{1/2}xg^{1/2}Jg\big)=g^{1/2}xg^{1/2}.$$ Finally, note that the range is contained in  $g^{1/2}\mathfrak{sp}_2(\mathcal{H})_hg^{1/2}$; $$\frac{1}{2}(x+gJxJg)=g^{1/2}\frac{1}{2}\bigg(g^{-1/2}xg^{-1/2}+g^{1/2}JxJg^{1/2}\bigg)g^{1/2}.$$  
To conclude we must show that the expression in the bracket anti-commutes with $J$, here we will use that $J^2=-1$ and the relation $g^{1/2}J=Jg^{-1/2}:$ 
$$\big(g^{-1/2}xg^{-1/2}+g^{1/2}JxJg^{1/2}\big)J=-g^{-1/2}JJxJg^{1/2}-Jg^{-1/2}xg^{-1/2}=$$ $$=-J\big(g^{1/2}JxJg^{1/2}+g^{-1/2}xg^{-1/2}\big).$$

Now we will show that ${\Pi_g^*}=\Pi_{g^{-1}}$; first note that if $x,y \in H_{\mathbb{R}}$ by the invariant and cyclic properties of the trace we have 
$$Tr(ygJxJg)=Tr(-JygJxJgJ)=Tr(JygJxg^{-1})=Tr(g^{-1}JygJx)$$ $$=Tr(g^{-1}JyJg^{-1}x).$$ Then the inner product is
 $$<\Pi_g(x),y>=Tr\bigg(y\big(\frac{1}{2}(x+gJxJg)\big)\bigg)=\frac{1}{2}Tr\big(yx+ygJxJg\big)=$$ $$=\frac{1}{2}\big(Tr(yx)+Tr(g^{-1}JyJg^{-1}x)\big).$$ On the other hand, we have $$ <x,\Pi_{g^{-1}}(y)>=Tr\bigg(\frac{1}{2}\big(y+g^{-1}JyJg^{-1}\big)x\bigg)=\frac{1}{2}\big(Tr(yx)+Tr(g^{-1}JyJg^{-1}x)\big).$$
 \end{proof}

\subsubsection{Linear connection and geodesics}

It is natural to consider a Hilbert-Riemann metric in ${\rm Sp}_2^+({\mathcal H})$, which consists of endowing each tangent space with the trace inner product. Therefore the Levi-Civita connection of this metric is given by differentiating in the ambient space $H_\mathbb{R}$ and projecting onto $T{\rm Sp}_2^+({\mathcal H})$. For this, we define the positive ambient metric as; $$\mathfrak{p}_{amb}(g,x):=\Vert x\Vert_2  $$ were $x \in T_g{\rm Sp}_2^+({\mathcal H})$. 
Using the formula of the projector over its range and Lemma \ref{adj}, we can calculate the orthogonal projection onto $T_g{\rm Sp}_2^+({\mathcal H})$; that is $$E_ {T_g{\rm Sp}_2^+({\mathcal H})}=\Pi_g(\Pi_g+\Pi_g^*-1)^{-1}=(\Pi_g+\Pi_g^*-1)^{-1}\Pi_g^*=(\Pi_g+\Pi_{g^{-1}}-1)^{-1}\Pi_{g^{-1}}.$$ Then, if $\gamma$ is a smooth curve in ${\rm Sp}_2^+({\mathcal H})$ and $\mathcal{X}(t)$ is a smooth tangent field along $\gamma$ the covariant derivative is $$\frac{D}{dt}\mathcal{X}(t) =E_{\gamma(t)}(\dot{\mathcal{X}}(t)).$$

\begin{prop} A curve $\alpha$ is a geodesic of the Levi-Civita connection if and only if it satisfies the differential equation $$\alpha\ddot{\alpha}\alpha+J\ddot{\alpha}J=0.$$ 
\end{prop}

\begin{proof} Using the last expression of the orthogonal projection $E$, we have $$\frac{D}{dt}\dot{\alpha}(t)=0 \Leftrightarrow \Pi_{{\alpha}^{-1}(t)}(\ddot{\alpha}(t))=0 \Leftrightarrow \ddot{\alpha}+\alpha^{-1}J\ddot{\alpha}J\alpha^{-1}=0.$$
\end{proof}

The study of this equation will appear elsewhere.

\subsubsection{Completeness of ${\rm Sp}_2^+({\mathcal H})$ with the geodesic distance}
Here we study the completeness of the metric space $({\rm Sp}_2^+({\mathcal H}),d_{{\mathfrak{p}_{amb}}})$. It is easy to verify that if we have any curve $\gamma\subset {\rm Sp}_2^+({\mathcal H})$ that joins $a$ to $b$, then $$\|a-b\|_2\leq \int_0^1 \|\dot{\gamma}(t)\|_2 dt=L_{{\mathfrak{p}_{amb}}}(\gamma).$$
From this inequality we have that 
\begin{equation}
\|a-b\|_2\leq d_{{\mathfrak{p}_{amb}}}(a,b),  \ \ \mbox{ for all } a,b \in {\rm Sp}_2^+({\mathcal H}). \label{desdistgeo}
\end{equation}
The key to prove the completeness will be the Proposition \ref{sp2cerrado} and the existence of smooth sections $\sigma_a$.  

\begin{prop} The metric space $({\rm Sp}_2^+({\mathcal H}),d_{{\mathfrak{p}_{amb}}})$ is complete.
\end{prop}
\begin{proof} Let $(x_n)$ be a Cauchy sequence for the metric $d_{{\mathfrak{p}_{amb}}}$, from equation (\ref{desdistgeo}) we have that $(x_n)$ is a Cauchy sequence in $\|.\|_2$ and then from Proposition \ref{sp2cerrado} we can take $x \in {\rm Sp}_2^+({\mathcal H})$ such that $\|x_n-x\|_2 \rightarrow 0$. Using the continuity of the global section $\sigma_x$, we have that $\|\sigma_x(x_n)-1\|_2=\|\sigma_x(x_n)-\sigma_x(x)\|_2\rightarrow 0$. For $n$ large we can take $z_n \in \mathfrak{sp}_2(\mathcal{H})$ such that $\sigma_x(x_n)=e^{z_n}$ and then it is clear using the previous fact that $\|z_n\|_2\rightarrow 0$. Let $\gamma_n(t)=e^{tz_n}xe^{t{z_n}^*}$ be a curve in ${\rm Sp}_2^+({\mathcal H})$ that joins $\gamma_n(0)=x$ and $\gamma_n(1)=e^{z_n}xe^{{z_n}^*}=\pi_x(\sigma_x(x_n))=x_n$; then if we compute its length, $$L_{{\mathfrak{p}_{amb}}}(\gamma_n)=\int_0^1 \|\dot{\gamma}_n(t)\|_2 dt=\int_0^1\|z_ne^{tz_n}xe^{tz_n^*}+e^{tz_n}xz_n^*e^{tz_n^*}\|_2 dt $$ $$\leq 2\|z_n\|_2\|x\|e^{\|z_n\|}\rightarrow 0.$$    
Then $$d_{{\mathfrak{p}_{amb}}}(x_n,x)\leq L_{{\mathfrak{p}_{amb}}}(\gamma_n) \rightarrow 0.$$ 
\end{proof}

\section{A Polar Riemannian structure}
The polar decomposition of $g \in {\rm Sp}_2(\mathcal{H})$ induces a diffeomorphism $${\rm Sp}_2(\mathcal{H})\stackrel{\varphi}\longrightarrow U_2(\mathcal{H}_J) \times {\rm Sp}_2^+({\mathcal H}) , \  g \longmapsto (u,\vert g \vert).$$ This fact was noted in Prop.14 (iv) on page 98 of the book \cite{Harpe}. The unitary group  $U_2(\mathcal{H}_J)$ is a Riemannian manifold with the metric given by the trace. We can endow the product manifold $U_2(\mathcal{H}_J) \times {\rm Sp}_2^+({\mathcal H})$ with the usual product metric, that is: if $v=(x,y) \in T_u U_2(\mathcal{H}_J)\times T_{\vert g\vert}{\rm Sp}_2^+({\mathcal H})$ we put 
$$\mathcal{P}\big((u,\vert g\vert),v\big):=\bigg(\Vert u^{-1}x\Vert_2^2 + \mathfrak{p}(\vert g\vert,y)^2\bigg)^{1/2}$$             
\begin{equation}  
      =\bigg(\Vert x\Vert_2^2+ \Vert \vert g\vert^{-1/2}y\vert g\vert^{-1/2}\Vert_2^2\bigg)^{1/2}.  \label{metpolar}
\end{equation}
This is the product metric in the Riemannian manifold $U_2(\mathcal{H}_J) \times {\rm Sp}_2^+({\mathcal H})$. The map $\varphi$ is an immersion, from this we can define a new Riemannian metric in the group in the following way: if $v,w \in (T{\rm Sp}_2(\mathcal{H}))_g$ we put $$\langle v,w\rangle_g:=\langle d\varphi_g(v) , d\varphi_g(w)\rangle_{(u,\vert g \vert)}.$$ It is clear that $\varphi$ is an isometric map with the above metric and if $\alpha$ is any curve in the group ${\rm Sp}_2(\mathcal{H})$ we can measure its length as $L_{\mathcal{P}}(\varphi\circ\alpha)$.        

\begin{teo} \label{minpolar} Let $g\in {\rm Sp}_2(\mathcal{H})$ with polar decomposition $u\vert g\vert$ and suppose that $u=e^x$ with $x \in \mathfrak{sp}_2(\mathcal{H})_{ah}$ and $\Vert x\Vert\leq \pi$, then the curve $\alpha(t)=e^{tx}\vert g\vert^t\subset {\rm Sp}_2(\mathcal{H})$ has minimal length among all curves joining $1$ to $g$, if we endow ${\rm Sp}_2(\mathcal{H})$ with the polar Riemannian metric (\ref{metpolar}).
\end{teo}

\begin{proof} By the polar decomposition, $\varphi\circ \alpha(t)=(e^{tx},\vert g\vert^t)$ and its length is $$L_{\mathcal{P}}(\varphi\circ \alpha)=\int_0^{1}        \mathcal{P}\big((e^{tx},\vert g\vert^t),(xe^{tx},\ln \vert g\vert \vert g\vert^{t})\big)dt=\big({\| x \|}^{2}_2 + {\|\ln \vert g\vert\|}^{2}_2\big)^{1/2}.$$ Let $\beta$ be another curve that joins the same endpoints and suppose that $\beta=\beta_1\beta_2$ is its polar decomposition where  $\beta_1\subset U_2(\mathcal{H}_J)$ and $\beta_2\subset {\rm Sp}_2^+({\mathcal H})$, then $$L_{\mathcal{P}}(\varphi\circ \beta)=\int_0^{1} \mathcal{P}\big((\beta_1,\beta_2), (\dot{\beta_1}, \dot{\beta_2})\big)dt = \int_0^{1} \big( \|\dot{\beta_1}\|_2^{2} + \mathfrak{p}(\beta_2,\dot{\beta_2})^{2}\big)^{1/2} dt. $$ Using the Minkowski inequality (see inequality 201 of \cite{Hardy}) we have, $$ \int_0^{1} \bigl( \|\dot{\beta_1}\|_2^{2} + \mathfrak{p}(\beta_2,\dot{\beta_2})^{2} \bigl)^{1/2} dt \geq \biggl( \biggl\lbrace\int_0^{1} \|\dot{\beta_1}\|_2\biggl\rbrace^{2} + \biggl\lbrace\int_0^{1} \mathfrak{p}(\beta_2,\dot{\beta_2})\biggl\rbrace^{2}\biggl)^{1/2}$$ $$=\biggl(L_2(\beta_1)^{2}+L_{\mathfrak{p}}(\beta_2)^{2}\biggl)^{1/2}.$$ It is know that the geodesic curve $e^{tx}$ has minimal length among all smooth curves in 
$U_2(\mathcal{H}_J)$ joining the same endpoints (see \cite{Andruchow2}); using this fact and from Theorem \ref{minpos} we have, $$L_2(\beta_1)\geq L_2(e^{tx})=\|x\|_2 \ \ \mbox{and} \ \ L_{\mathfrak{p}}(\beta_2)\geq L_{\mathfrak{p}}(e^{t\ln(\vert g \vert})=\|\ln \vert g \vert\|_2$$ then it is clear that $L_{\mathcal{P}}(\varphi\circ \beta)\geq L_{\mathcal{P}}(\varphi\circ \alpha)$.
\end{proof}

\begin{rem} Let $p,q \in {\rm Sp}_2(\mathcal{H})$, suppose that $u_p\vert p\vert$ and $u_q\vert q\vert$ are their polar decompositions, from the surjectivity of the exponential map we can choose $z \in \mathfrak{sp}_2(\mathcal{H})_{ah}$ such that $u_q=u_pe^{z}$ with $\|z\|\leq \pi$, then the curve
$$\alpha_{p,q}(t)=u_pe^{tz}\vert p\vert^{1/2}({\vert p\vert}^{-1/2}\vert q\vert {\vert p\vert}^{-1/2})^t\vert p\vert^{1/2}\subset  {\rm Sp}_2(\mathcal{H})$$ has minimal length among all curves joining $p$ to $q$.  
\end{rem}
The above fact shows that the curve $\alpha_{p,q}$ is a geodesic of the Levi-Civita connection of the polar metric. Its length is $$\bigg({\|z\|}^{2}_2 + {\|\ln \vert p\vert^{-1/2} \vert q\vert\vert p\vert^{-1/2}\|}^{2}_2\bigg)^{1/2}. $$ From this, the geodesic distance is
$$d_{\mathcal{P}}(p,q)=\big(d_2(u_p,u_q)^2+d_{\mathfrak{p}}(\vert p\vert,\vert q \vert)^2\big)^{1/2}.$$

\paragraph{Special case: normal speed.} 
If the initial condition $v \in \mathfrak{sp}_2(\mathcal{H})$ is normal, then the geodesics starting at the identity map coincide with the geodesics from polar metric. Indeed, if $v=x+y$ is the decomposition in $\mathfrak{sp}_2(\mathcal{H})_h\oplus \mathfrak{sp}_2(\mathcal{H})_{ah}$ and $v$ is normal a straightforward computation shows that $x$ commutes with $y$, thus we have $$e^{tv^{*}}e^{t(v-v^{*})}=e^{tv}=e^{tx}e^{ty}.$$This equation shows that the geodesic are one-parameter groups when the initial speed is normal.    

\begin{prop} The metric space $( {{\rm Sp}_2(\mathcal{H})},d_\mathcal{P})$ is complete.
\end{prop}

\begin{proof} Let $(x_n)\subset {\rm Sp}_2(\mathcal{H})$ be a Cauchy sequence with $d_{\mathcal{P}}$, if $x_n=u_{x_n}\vert x_n\vert$ is its polar decomposition, we have that $$d_2(u_{x_n},u_{x_m})\leq d_{\mathcal{P}}(x_n,x_m)=\big(d_2(u_{x_n},u_{x_m})^2+d_{\mathfrak{p}}(\vert x_n\vert,\vert x_m \vert)^2\big)^{1/2}$$ then the unitary part is a Cauchy sequence in $(U_2(\mathcal{H}_J),d_2)$  and by \cite{Andruchow2} it is  $d_2$ convergent to an element $u \in U_2(\mathcal{H}_J)$. Analogously the positive part is a Cauchy sequence in $({\rm Sp}_2^+({\mathcal H}),d_{\mathfrak{p}})$ then it is convergent to an element $g\in {\rm Sp}_2^+({\mathcal H})$. If we put $x:=ug \in {\rm Sp}_2(\mathcal{H})$  then, $$d_{\mathcal{P}}(x_n,x)=\big(d_2(u_{x_n},u)^2+d_{\mathfrak{p}}(\vert x_n\vert,g)^2\big)^{1/2} \rightarrow 0.$$         
\end{proof}

In the next steps we will compare the geodesic distance measured with the polar metric versus the left invariant metric. To do it we need the following proposition first. 

\begin{prop} \label{despol} Given $p,q \in {\rm Sp}_2(\mathcal{H})$, if we denote $v:=\vert p\vert^{-1/2} \vert q\vert\vert p\vert^{-1/2}$ we can estimate the geodesic distance $d_{\mathcal{I}}$ by the geodesic distance $d_{\mathcal{P}}$ as, $$d_{\mathcal{I}}(p,q)\leq c(p,q)d_{\mathcal{P}}(p,q)$$ where $$c(p,q)^2=2\max\big\lbrace \ e^{4\|\ln(v)\|}  \big(\| p \|  \| p^{-1} \|\big)^2 ,   \| p \|  \| p^{-1}\| \big\rbrace.$$  
\end{prop}
\begin{proof} The proof consist of estimate $L_{\mathcal{I}}(\alpha_{p,q})$; if we derive ${\alpha}_{p,q}$ we have, $$\dot{\alpha}_{p,q}=u_pze^{tz}\vert p\vert^{1/2}e^{t\ln(v)}\vert p\vert^{1/2}+u_pe^{tz}\vert p\vert^{1/2}\ln(v)e^{t\ln(v)}\vert p\vert^{1/2}$$ and the inverse of the curve ${\alpha}_{p,q}$ is $$\alpha^{-1}_{p,q}=\vert p\vert^{-1/2}e^{-t\ln(v)}\vert p\vert^{-1/2}e^{-tz}u_p^{-1}.$$ After some simplifications we can write $$\alpha^{-1}_{p,q}\dot{\alpha}_{p,q}=\vert p\vert^{-1/2}e^{-t\ln(v)}\vert p\vert^{-1/2}z\vert p\vert^{1/2}e^{t\ln(v)}\vert p\vert^{1/2}+\vert p\vert^{-1/2}\ln(v)\vert p\vert^{1/2}.$$ Let  $x:=\vert p\vert^{1/2}e^{t\ln(v)}\vert p\vert^{1/2},$ taking the norm and using the parallelogram rule we have, 
\begin{align}
\|\alpha^{-1}_{p,q}\dot{\alpha}_{p,q}\|^{2}_2 &= \|x^{-1}zx+\vert p\vert^{-1/2}\ln(v)\vert p\vert^{1/2}\|^{2}_2  \nonumber\\
& \leq 2\big (\|x^{-1}zx\|^{2}_2+\|\vert p\vert^{-1/2}\ln(v)\vert p\vert^{1/2}\|^{2}_2\big)  \nonumber\\
& \leq 2\big (\|x^{-1}\|^2 \ \|x\|^2 \ \|z\|^{2}_2+ \| \vert p \vert^{-1/2} \|^2 \ \| \ln(v) \|^{2}_2 \ \| \vert p\vert^{1/2}\|^2\big ) \label{des} .
\end{align}
We can estimate  $\|x\|^2$ and $\|x^{-1}\|^2$ by $$\|x\|^2\leq \| \vert p\vert^{1/2}\|^4 \ e^{2\|\ln(v)\|}=\|p\|^2e^{2\|\ln(v)\|} $$
 and $$\|x^{-1}\|^2\leq \| \vert p\vert^{-1/2}\|^4 \ e^{2\|\ln(v)\|}=\|p^{-1}\|^2e^{2\|\ln(v)\|}.$$
If we define  $$c(p,q)^2=2\max \big\lbrace \ e^{4\|\ln(v)\|}  \big(\| p \|  \| p^{-1} \|\big)^2 ,   \| p \|  \| p^{-1}\| \big\rbrace. $$
from (\ref{des}) and taking square roots we have, $$\|\alpha^{-1}_{p,q}\dot{\alpha}_{p,q}\|_2\leq c(p,q)\big(\|z\|^{2}_2+\| \ln(v) \|^{2}_2\big)^{1/2}=
c(p,q)d_{\mathcal{P}}(p,q),$$ then $$d_{\mathcal{I}}(p,q)\leq L_{\mathcal{I}}(\alpha_{p,q})\leq c(p,q)d_{\mathcal{P}}(p,q).$$

\end{proof}

\section{The metric space $({\rm Sp}_2(\mathcal{H}),d_{\mathcal{I}})$}
In this section we will prove the main result of this paper, that is the completeness of $({\rm Sp}_2(\mathcal{H}),d_{\mathcal{I}})$, it will be deduced from the completeness of $(U_2(\mathcal{H}_J),d_2)$ and from Proposition \ref{despol}. The next lemma is essential for the proof.

\begin{lem} If $(x_n)\subset{\rm Sp}_2(\mathcal{H})$ is a Cauchy sequence in $({\rm Sp}_2(\mathcal{H}),d_{\mathcal{I}})$ then it is a Cauchy sequence in  $({\rm Sp}_2(\mathcal{H}),\|. \|_2)$.
\end{lem}
\begin{proof} First we take $W,U$ geodesic neighboords of $0$ and $1$ respectively such that  $$Exp_1 : W \longrightarrow U:=Exp_1(W)\subset {\rm Sp}_2(\mathcal{H})$$ is a diffeomorphism. If $(x_n)$ is  $d_{\mathcal{I}}$-Cauchy, given small $\varepsilon$ there exist $n(\varepsilon)$ such that $d_{\mathcal{I}}(x_n^{-1}x_{n+p},1)=d_{\mathcal{I}}(x_{n+p},x_n)<\varepsilon$ $\forall p$. Then we can suppose that $x_n^{-1}x_{n+p} \in U$  for all $p$. Let $\alpha_p(t)=e^{tv_p^{*}}e^{t(v_p-v_p^{*})}=Exp_1(tv_p)$ with $v_p \in W$ be the minimal curve that joins $1$ to $x_n^{-1}x_{n+p}$, then    
$$d_{\mathcal{I}}(x_n^{-1}x_{n+p},1)=L_{\mathcal{I}}(\alpha_p)=\| v_p\|_2<\varepsilon.$$ We have 
$$\|x_n^{-1}x_{n+p}-1\|_2\leq \int_0^{1} \|\dot{\alpha_p}(t)\|_2 dt  \leq \int_0^{1} \|\alpha_p(t)\| \|\alpha_p^{-1}\dot{\alpha_p}(t)\|_2 dt,$$
$$\|\alpha_p(t)\|=\| e^{tv_p^{*}}e^{t(v_p-v_p^{*})}\|\leq e^{3\| v_p\|_2}\leq e^{3\varepsilon}.$$ From this, $$\|x_n^{-1}x_{n+p}-1\|_2\leq e^{3\varepsilon}\varepsilon, \ \mbox{for all} \ p.$$ This fact shows that the sequence is bounded in the uniform norm; indeed if we take $\varepsilon_0$ such that the sequence belongs in the geodesic neighboord $U$, then there exists $n_0$ (fixed) such that $\|x_{n_0}^{-1}x_{n_0+p}-1\|_2\leq e^{3\varepsilon_0}\varepsilon_0, \ \mbox{for all} \ p$. Then if $m=n_0+p>n_0$, we have $$ \vert \|x_{n_0} \|- \|x_m \|\vert\leq   \|x_{n_0}-x_m \|_2 \leq  \|x_{n_0} \|\|x_{n_0}^{-1}x_{n_0+p}-1\|_2\leq   \|x_{n_0} \|e^{3\varepsilon_0}\varepsilon_0 ;$$ then $$  \|x_{m=n_0+p} \|\leq \vert \|x_m \|- \|x_{n_0} \|\vert +  \|x_{n_0} \|\leq \|x_{n_0} \|(1+e^{3\varepsilon_0}\varepsilon_0) \  \forall p.$$ To complete the proof, if $n$ is large, we have $$\|x_{n+p}-x_n\|_2=\|x_n(x_n^{-1}x_{n+p} -1)\|_2\leq \|x_n\|e^{3\varepsilon}\varepsilon\leq Ke^{3\varepsilon}\varepsilon  \ \ \forall p.$$    
\end{proof}
Now we are in a position to obtain our main result.
\begin{teo} The metric space $({\rm Sp}_2(\mathcal{H}),d_{\mathcal{I}})$ is complete.
\end{teo}
\begin{proof} Let $(x_n)\subset {\rm Sp}_2(\mathcal{H})$ be a $d_{\mathcal{I}}$-Cauchy sequence, by the above lemma it is $\Vert .\Vert_2$-Cauchy; then from Proposition \ref{sp2cerrado} there exists $x \in {\rm Sp}_2(\mathcal{H})$ such that $x_n \stackrel{\|.\|_2}\longrightarrow x$. Now we will show that $x_n \stackrel{d_{\mathcal{P}}} \longrightarrow x$; indeed from the continuity of the module we have that $\vert x_n \vert$ converges to $\vert x\vert$ in $\|.\|_2$ and its unitary part $u_{x_n}=x_n\vert x_n\vert^{-1}$ converges to $u_x=x\vert x\vert^{-1}$. The sequence $\vert x\vert^{-1/2}\vert x_n\vert\vert x\vert^{-1/2}$ converges to $1$ and then the geodesic distance $d_{\mathfrak{p}}(\vert x_n\vert,\vert x\vert)
=\|\ln(\vert x\vert^{-1/2}\vert x_n\vert\vert x\vert^{-1/2})\|_2\rightarrow 0.$ By the equivalence of metrics in $U_2(\mathcal{H}_J)$ (see \cite{Andruchow2} for a proof) we have $$\sqrt{1-\dfrac{\pi^2}{12}}d_2(u_{x_n},u_x)\leq \|u_{x_n}-u_x\|_2\leq d_2(u_{x_n},u_x)$$ and then 
 $$d_\mathcal{P}(x_n,x)=\big(d_2(u_{x_n},u_x)^{2}+d_{\mathfrak{p}}(\vert x_n\vert,\vert x\vert)^{2}\big)^{1/2} \longrightarrow 0.$$ From Proposition \ref{despol} we have $d_{\mathcal{I}}(x,x_n)\leq c(x,x_n)d_\mathcal{P}(x,x_n)$; now we will see that $c(x,x_n)$ is uniformly bounded. Indeed, for $n$ large we can suppose that $\vert x_n\vert \leq \vert x\vert +1$ then according the notation of Proposition \ref{despol} $$v_n=\vert x\vert^{-1/2}\vert x_n\vert \vert x \vert^{-1/2}\leq \vert x\vert^{-1/2}(\vert x\vert +1)\vert x\vert^{-1/2}=1+\vert x\vert^{-1}$$ and by the monotonicity of the logarithm $\ln(v_n)\leq \ln(1+\vert x\vert^{-1})$ and then $\|\ln(v_n)\|\leq \|\ln(1+\vert x\vert^{-1}) \|$ for all $n$. Finally we have $$c(x,x_n)^2=2\max \lbrace 
   \ e^{4\|\ln(v_n)\|}   \big( \|x\| \|x^{-1}\|\big)^2, \ \|x\| \|x^{-1}\| \rbrace$$ $$\leq 2\max \lbrace  \ e^{4\|\ln(1+\vert x\vert^{-1})\|} \big( \|x\| \|x^{-1}\|\big)^2 , \ \|x\| \|x^{-1}\| \rbrace$$ is clearly uniformly
bounded and then it is clear that $d_{\mathcal{I}}(x,x_n)\rightarrow 0$.
   
\end{proof}
\paragraph{Acknowledgements}
I want to thank Prof. E. Andruchow and Prof. G. Larotonda for their suggestions and support.

\bigskip
{\footnotesize Manuel L\'opez Galv\'an.\\
Instituto de Ciencias, Universidad Nacional de General Sarmiento.\\
JM Guti\'errez 1150 (1613) Los Polvorines. Buenos Aires, Argentina.\\
e-mail: mlopezgalvan@hotmail.com}


\begin{thebibliography}{99}

\bibitem{Andruchow1} E. Andruchow, G. Larotonda, L. Recht, A. Varela. The left invariant metric in the general linear group (2014). Journal of Geometry and Physics, volume 86 (2014), 241-257.

\bibitem{Andruchow2} E. Andruchow, G. Larotonda. Hopf-Rinow Theorem in the Sato Grassmanian. J. Funct. Anal. 255 (2008) no.7, 1692-1712.

\bibitem{Atkin} C.J. Atkin. The Hopf-Rinow theorem is false in infinite dimensions. Bull. London Math. Soc. 7 (1975), 261-266.

\bibitem{Beltita} D. Belti\c{t}$\check{a}$. Smooth homogeneous structures in operator theory. Chapman and Hall/CRC. Monographs and Surveys in Pure and Applied Mathematics, 137. Chapman and Hall/CRC, Boca Raton, FL, 2006.

\bibitem{Cook} J. M. Cook. The mathematics of second quantization. Trans. Amer. Math. Soc. 74 (1953), 222-245.

\bibitem{Corach3} G. Corach. Sobre la geometr\'ia del conjunto de operadores positivos en espacios de Hilbert, Anal. Acad. Nac. Cs. Ex. Fis. y Nat., Buenos Aires, Argentina, 50 (1998), 109-118.

\bibitem{Corach1} G. Corach, H. Porta, L. Recht. Geodesics and operator means in the space of positive operators. Internat. J. Math. 4 (1993), 193-202. 

\bibitem{Corach2} G. Corach, H. Porta, L. Recht. Convexity of the geodesic distance on the space of positive operators. Illions J.Math. 38 (1994) no. 1, 87-94.

\bibitem{Harpe} Pierre de la Harpe. Classical Banach-Lie Algebras and Banach-Lie Groups of Operators in Hilbert Space. Springer-Verlag. Berlin. Heidelberg. NewYork 1972.

\bibitem{Hardy} G.H. Hardy, J.E. Littlewood, G. Polya, Inequalities, Cambridge University. Press, London, 1934.

\bibitem{Hiai} F. Hiai, H. Kosaki. Comparison of various means of operators. J. Funct. Anal. 163 (1999), no 2, 300-323. 

\bibitem{Lang} S. Lang. Differentiable and Riemannian manifolds. Third edition. Graduate Texts in Mathematics, 160. Springer-Verlag, New York, 1995.

\bibitem{Larotonda} G. Larotonda. Nonpositive Curvature: A Geometric Approach to Hilbert-Schmidt Operators. Differential Geom. Appl. 25 (2007) no. 6, 679-700. 

\bibitem{Larotonda2} G. Larotonda. Estructuras geom\'etricas para las Variedades de  Banach. Colecci\'on Ciencia Innovaci\'on y Desarrollo, Univ. Nac. de Gral. Sarmiento, 2012.  

\bibitem{McAlpin} J. McAlpin. Infinite dimensional manifolds and Morse theory. Thesis, Columbia University, 1965.

\bibitem{Mostow} G.D. Mostow. Some new decompositions theorems for semi-simple groups. Mem. Amer. Math. Soc. 14 (1955), 31-54.


\bibitem{Shale} David Shale. Linear Symmetries of Free Boson Field. Transactions of the American Mathematical Society, Vol. 103, No.1 (Apr. 1962), 149-167.    

\end{thebibliography}
\end{document}